\UseRawInputEncoding
\documentclass[11pt, reqno]{article}
\usepackage{amssymb}
\usepackage{amsthm}
\usepackage{amsthm,amsmath,indentfirst,amssymb}
\usepackage{pstricks}
\usepackage{algorithm} 
\usepackage{algorithmic} 
\usepackage{blindtext}
\usepackage{graphicx}
\usepackage{cite}
\usepackage{float}
\usepackage[scriptsize,nooneline]{subfigure}
\usepackage{epstopdf}
\usepackage[colorlinks,
linkcolor=blue,
anchorcolor=blue,
citecolor=blue
]{hyperref}
\def\({\left(}
\def \){ \right)}
\def\lf{\left}
\def\r{\right}
\newcommand{\R}{{\mathbb R}}
\newcommand{\M}{{\mathcal M}}
\newtheorem{definition}{Definition}[section]
\newtheorem{theorem}{Theorem}[section]
\newtheorem{lemma}{Lemma}[section]

\newtheorem{remark}{Remark}[section]

\begin{document}
\title{The K\"othe dual of mixed Morrey spaces and applications
\footnotetext{\hspace{-0.35cm} 2020{\emph{
Mathematics Subject Classification}}. 42B35, 30H35, 46E30, 42B20, 42B25.
\endgraf This project is supported by the National Natural Science Foundation of China (Grant No.12061069).}}
\author{Houkun Zhang,
Jiang Zhou\,\footnote{Corresponding author
E-mail address: \texttt{Zhoujiang@xju.edu.cn}.
}
\\[.5cm]
}
\date{}
\maketitle
{\bf Abstract:}\quad In this paper, we study the separable and weak convergence of mixed-norm Lebesgue spaces. Furthermore, we prove that the block space $\mathcal{B}_{\vec{p}\,'}^{p'_0}(\mathbb{R}^n)$ is the K\"othe dual of the mixed Morrey space $\mathcal{M}_{\vec{p}}^{p_0}(\mathbb{R}^n)$ by the Fatou property of these block spaces. The boundedness of the Hardy--Littlewood maximal function is further obtained on the block space $\mathcal{B}_{\vec{p}\,'}^{p'_0}(\mathbb{R}^n)$.  As applications, the characterizations of $BMO(\R^n)$ via the commutators of the fractional integral operator $I_{\alpha}$ on mixed Morrey spaces are proved as well as the block space $\mathcal{B}_{\vec{p}\,'}^{p'_0}(\mathbb{R}^n)$.

\maketitle

\section{Introduction}\label{sec1}
\par
In recent years, since the precise structure of mixed-norm function spaces, the mixed-norm function spaces are widely used in the partial differential equations \cite{3,4,5,6}. In 1961, the mixed-norm Lebesgue space $L^{\vec{p}}(\mathbb{R}^n) (0<\vec{p}\le\infty)$ were studied by Benedek and Panzone \cite{7}. These spaces are natural generalizations of the classical Lebesgue space $L^p(\R^n)(0<p\le\infty)$. After that, many function spaces with mixed norm were introduced, such as mixed-norm Lorentz spaces \cite{8}, mixed-norm Lorentz-Marcinkiewicz spaces \cite{9}, mixed-norm Orlicz spaces \cite{10}, anisotropic mixed-norm Hardy spaces \cite{11}, mixed-norm Triebel-Lizorkin spaces \cite{12} and weak mixed-norm Lebesgue spaces \cite{13}. More studies can be refereed in \cite{14} and so on.

In 2019, Nogayama \cite{N2019(2),15} introduced mixed Morrey spaces
associated with mixed-norm Lebesgue spaces and Morrey spaces, and further proved that the block space $\mathcal{B}_{\vec{p}\,'}^{p'_0}(\mathbb{R}^n)$ are predual spaces of the mixed Morrey space $\mathcal{M}_{\vec{p}}^{p_0}(\R^n)$. In this paper, we will prove that $\mathcal{B}_{\vec{p}\,'}^{p'_0}(\mathbb{R}^n)$ are also K\"othe duals of these mixed Morrey spaces. For more studies and a deeper account of developments about the K\"othe dual we may consult \cite{18} and the references therein.

Given $0<\alpha<n$, for a measurable function $f$ on $\mathbb{R}^n$, the fractional integral operators are defined by
$$I_\alpha f(x)=\int_{\mathbb{R}^n}\frac{f(y)}{|x-y|^{n-\alpha}}dy,$$
and these operators play such a prominent role in real and harmonic analysis \cite{1,2}. By simple calculations,
\begin{equation}\label{eq1.1}
\int_{\mathbb{R}^n}f(x)I_{\alpha}g(x)dx=\int_{\mathbb{R}^n}g(x)I_{\alpha}f(x)dx
\end{equation}
can be obtained.

For a locally integrable function $b$ and a measurable function $f$, the commutator of the fractional integral operators is defined by
$$[b,I_\alpha]f(x):=b(x)I_\alpha f(x)-I_\alpha(bf)(x)=\int_{\mathbb{R}^n}\frac{(b(x)-b(y))f(y)}{|x-y|^{n-\alpha}}dy,$$
which was introduced by Chanillo in \cite{19}. For the commutators, a similar formula
\begin{equation}\label{eq1.2}
\int_{\mathbb{R}^n}f(x)[b,I_{\alpha}]g(x)dx=-\int_{\mathbb{R}^n}g(x)[b,I_{\alpha}]f(x)dx
\end{equation}
can be obtained, where $f,g$ are measurable functions.

Let us recall the classical results. In 1991, Di Fazio and Ragusa given the characterizations of $BMO(\mathbb{R}^n)$ via the boundedness of $[b,I_{\alpha}]$ on classical Morrey spaces \cite{20}. That is, let $1<p\le p_0<\infty$, $1<q\le q_0<\infty$, $\frac{p}{p_0}=\frac{q}{q_0}$, and $\frac{1}{p_0}-\frac{1}{q_0}=\frac{\alpha}{n}$. Then
$$b\in BMO(\mathbb{R}^n)\Rightarrow [b,I_{\alpha}]:\mathcal{M}_{p}^{p_0}(\mathbb{R}^n)\mapsto\mathcal{M}_{q}^{q_0}(\mathbb{R}^n).$$
Conversely, if $n-\alpha$ is an even integer, then
$$[b,I_{\alpha}]:\mathcal{M}_{p}^{p_0}(\mathbb{R}^n)\mapsto\mathcal{M}_{q}^{q_0}(\mathbb{R}^n)\Rightarrow b\in BMO(\mathbb{R}^n).$$
In 2006, Shirai given another characterizations of $BMO(\mathbb{R}^n)$ via the boundedness of $[b,I_{\alpha}]$ on classical Morrey spaces \cite{21}. That is
$$b\in BMO(\mathbb{R}^n)\Leftrightarrow [b,I_{\alpha}]:\mathcal{M}_{p}^{p_0}(\mathbb{R}^n)\mapsto\mathcal{M}_{q}^{q_0}(\mathbb{R}^n),$$
where $1<p\le p_0<\infty$, $1<q\le q_0<\infty$ and $\frac{1}{p}-\frac{1}{q}=\frac{1}{p_0}-\frac{1}{q_0}=\frac{\alpha}{n}$.

In 2019, Nogayama proved characterizations of $BMO(\mathbb{R}^n)$ via the operator $[b,I_{\alpha}]$ on mixed Morrey spaces \cite{N2019(2)}. That is
$$b\in BMO(\mathbb{R}^n)\Leftrightarrow [b,I_{\alpha}]:\mathcal{M}_{\vec{p}}^{p_0}(\mathbb{R}^n)\mapsto\mathcal{M}_{\vec{q}}^{q_0}(\mathbb{R}^n),$$
where $1<\frac{1}{p_0}<\frac{\alpha}{n}$, $\frac{n}{p_0}\le\sum_{i=1}^{n}\frac{1}{p_i}$, $\frac{n}{q_0}\le\sum_{i=1}^{n}\frac{1}{q_i}$, $\frac{1}{p_0}-\frac{1}{q_0}=\frac{\alpha}{n}$ and $\frac{\vec{p}}{p_0}=\frac{\vec{q}}{q_0}$.

In fact, the results of \cite{N2019(2)} can be regarded as a generation of \cite{20}. In the paper, we generalize the results of \cite{21} on mixed Morrey spaces. We point out that
$$\frac{1}{p}-\frac{\alpha}{n}=\frac{p_0}{p}\left(\frac{1}{p_0}-\frac{p}{p_0}\frac{\alpha}{n}\right) \ge\frac{p_0}{p}\left(\frac{1}{p_0}-\frac{\alpha}{n}\right)=\frac{p_0}{pq_0}.$$
Thus, it's easy to see that the results of \cite{20} can be regard as the improved results of \cite{21}. But, for mixed Morrey spaces, we can only prove that
$$
\sum_{i=1}^{n}\frac{1}{p_i}-\alpha\ge\sum_{i=1}^{n}\frac{p_0}{p_i}\cdot\left(\frac{1}{p_0}-\frac{\alpha}{n}\right).
$$
In other words, the results of \cite{N2019(2)} are not the improvements of Theorem \ref{th6.3}. Furthermore, when $\frac{1}{p_0}=\frac{1}{n}\sum_{i=1}^{n}\frac{1}{p_i}$ and $\frac{1}{q_0}=\frac{1}{n}\sum_{i=1}^{n}\frac{1}{q_i}$, the results of Theorem \ref{th6.3} are better than \cite{N2019(2)}.

This paper is organized as the following. In Section \ref{sec2}, main definitions are recalled. In Section \ref{sec3}, separability and weak convergence of mixed-norm Lebesgue spaces are also studied. We prove that the K\"othe dual of the mixed Morrey space $\mathcal{M}_{\vec{p}}^{p_0}(\mathbb{R}^n)$ are the block space $\mathcal{B}_{\vec{p}\,'}^{p'_0}(\mathbb{R}^n)$ by the property of Fatou in Section \ref{sec4}. In Section \ref{sec5}, we prove the bounds for the Hardy--Littlewood maximal function over the block spaces. As an application, the characterizations of $BMO(\mathbb{R}^n)$ is given and the boundedness of $I_{\alpha}$ and $[b,I_{\alpha}]$ are also proved in Section \ref{sec6}.

Finally, we make some conventions on notation. Let $\vec{p}=(p_1,\cdots,p_n)$, $\vec{q}=(q_1,\cdots,q_n)$ are n-tuples with $1<p_i,q_i<\infty,~i=1,...,n$. $\vec{p}<\vec{q}$ means that $p_i<q_i$ holds, and $\frac{1}{\vec{p}}+\frac{1}{\vec{p}\,'}=1$ means $\frac{1}{p_i}+\frac{1}{p_i'}=1$ holds, for each $i=1,...,n$. The symbol $Q$ denotes the cubes whose edges are parallel to the coordinate axes and $Q(x,r)$ denotes a open cube centered at $x$ of side length $r$. Let $cQ(x,r)=Q(x,cr)$. Denote by the symbol $\mathfrak{M}(\mathbb{R}^n)$ the set of all measurable function on $\mathbb{R}^n$. $A\sim B$ means that $A$ is equivalent to $B$. That is $A\le CB$ and $B\le CA$, where $C$ is a positive constant. Through all paper, every positive constant $C$ is not necessarily equal.

\section{Main definitions}\label{sec2}
We begin this section with the definition of some maximal functions in the following.

For a locally integrable function $f$, the Hardy--Littlewood maximal operator is defined by, for almost every $x\in\R^n$,
$$M(f)(x):=\sup_{Q\ni x}\frac{1}{|Q|}\int_{Q}|f(y)|dy,$$
where the supremum is taken over all cubes $Q\subset\mathbb{R}^n$ containing $x$, and the sharp maximal operator is defined by
$$M^{\sharp}(f)(x):=\sup_{Q\ni x}\frac{1}{|Q|}\int_{Q}|f(y)-f_Q|dy,$$
where $f_Q=\frac{1}{|Q|}\int_{Q}f$ and the supremum is taken over all cubes $Q\subset\mathbb{R}^n$ containing $x$.

The definition of ball (quasi-)Banach function spaces is presented as follows, which were introduced by Sawano et al. \cite{23}.

\begin{definition} \label{def2.2}
A (quasi-)Banach space $\mathbf{X}\subset\mathfrak{M}(\mathbb{R}^n)$ with (quasi-)norm $\|\cdot\|_{\mathbf{X}}$ is called a ball (quasi-)Banach function space if
\begin{itemize}
\item[(\romannumeral1)] $|g|\le|f|$ almost everywhere implies that $\|g\|_{\mathbf{X}}\le\|f\|_{\mathbf{X}}$;
\item[(\romannumeral2)] $0\le f_m\uparrow f$ almost everywhere implies that $\|f_m\|_{\mathbf{X}}\uparrow\|f\|_{\mathbf{X}};$
\item[(\romannumeral3)] If $|Q|<\infty$, then $\chi_Q\in \mathbf{X}$;
\item[(\romannumeral4)] If $f\ge 0$ almost everywhere and $|Q|<\infty$, then
$$\int_Q f(x)dx\le C_Q\|f\|_{\mathbf{X}};$$
for some positive constants $C_Q$, $0<c_Q<\infty$, depending on $Q$ but independent of $f$.
\end{itemize}
\end{definition}

\begin{remark}\label{re2.3}
If "cube" is replaced by "ball" in the preceding definition, it's also valid. In particular, if we replace any cubes $Q$ by any measurable sets $E$ in Definition \ref{def2.2}, it is (quasi-)Banach function spaces (see \cite[Definition 1.1 of Chapter 1]{25}).
\end{remark}
The definition of the associate space of a ball (quasi-)Banach function space can be found in \cite[chapter 1]{25}
as follows.
\begin{definition} \label{def2.3}
For any ball (quasi-)Banach function spaces $\mathbf{X}$, the associate space (also called the K\"othe dual) $\mathbf{X}'$ is defined by setting
$$\mathbf{X}':=\left\{f\in \mathfrak{M}(\mathbb{R}^n):\|f\|_{\mathbf{X}'}: =\sup_{g\in\mathbf{X},\|g\|_{\mathbf{X}}=1}\int_{\mathbb{R}^n}|f(x)g(x)|dx<\infty\right\},$$
where $\|\cdot\|_{\mathbf{X}'}$ is called the associate norm of $\|\cdot\|_{\mathbf{X}}.$
\end{definition}

\begin{remark}\label{re2.4}
\begin{itemize}
\item[(\romannumeral1)] In other literatures (for example \cite{33}) the Banach function spaces and the associate spaces are called the K\"othe spaces and the K\"othe dual respectively. Thus, the associate spaces of ball (quasi-)Banach function spaces are called the K\"othe dual of ball (quasi-)Banach function spaces.
\item[(\romannumeral2)] It is easy to prove that
    $$\int_{\mathbb{R}^n}|f(x)g(x)|dx\le\|g\|_{\mathbf{X}}\|f\|_{\mathbf{X}'}.$$
\item[(\romannumeral3)] Due to \cite[Lemma 2.6]{34}, if $\mathbf{X}$ is a ball Banach function space, then
    $$\|f\|_{\mathbf{X}}=\|f\|_{\mathbf{X}''},$$
    where $\mathbf{X}''=(\mathbf{X}')'$.
\item[(\romannumeral4)] According to the Definition \ref{def2.3} and (\romannumeral4), it is easy to know that
    \begin{equation}\label{eq2.1}
    \|f\|_{\mathbf{X}'}=\||f|\|_{\mathbf{X}'}=\sup_{g\in\mathbf{X},\|g\|_{\mathbf{X}}=1}\int_{\mathbb{R}^n}|f(x)g(x)|dx
    \end{equation}
    and
    \begin{equation}\label{eq2.2}
    \|f\|_{\mathbf{X}}=\|f\|_{\mathbf{X}''}=\||f|\|_{\mathbf{X}''}=\||f|\|_{\mathbf{X}}.
    \end{equation}
    Furthermore, by (\ref{eq2.1}) and (\ref{eq2.2}), for $f\in X'$ and $g\in \mathbf{X}$,
\begin{align*}
\|f\|_{\mathbf{X}'}&=\sup_{\|g\|_{\mathbf{X}}=1}\int_{\mathbb{R}^n}|f(x)g(x)|dx
=\sup_{\|g\|_{\mathbf{X}}=1}\left|\int_{\mathbb{R}^n}f(x)g(x)dx\right|
=\sup_{\|g\|_{\mathbf{X}}=1}\int_{\mathbb{R}^n}f(x)g(x)dx.
\end{align*}
Indeed, we have
$$\int_{\mathbb{R}^n}f(x)g(x)dx\le\left|\int_{\mathbb{R}^n}f(x)g(x)dx\right|\le\int_{\mathbb{R}^n}|f(x)g(x)|dx,$$
and for $h=sgn(fg)|g|$,
$$\int_{\mathbb{R}^n}f(x)h(x)dx=\left|\int_{\mathbb{R}^n}f(x)h(x)dx\right|=\int_{\mathbb{R}^n}|f(x)g(x)|dx.$$
\end{itemize}
\end{remark}
We still recall the notion of the convexity of ball (quasi-)Banach function spaces, which can be found in \cite[Definition 2.6]{23}.
\begin{definition} \label{def2.5}
Let $\mathbf{X}$ be a ball (quasi-)Banach function space and $0<p<\infty$. The p-convexification $\mathbf{X}^p$ of $\mathbf{X}$ is defined by setting
$$\mathbf{X}^p:=\left\{f\in\mathfrak{M}(\mathbb{R}^n):|f|^p\in\mathbf{X}\right\}$$
equipped with the (quasi-)norm $\|f\|_{\mathbf{X}^p}:=\| |f|^p\|_{\mathbf{X}}^{\frac{1}{p}}$.
\end{definition}

Obviously, if $\mathbf{X}$ is a ball (quasi-)Banach function space, the $\mathbf{X}^p$ and $\mathbf{X}'$ are also ball (quasi-)Banach function spaces. Now, let us recall $A_{p}(\mathbb{R}^n)$-weight and weighted Lebesgue spaces.
\begin{definition}
Let $1<p<\infty$. A weight $w$ is said to be of class $A_p(\mathbb{R}^n)$ if
$$[\omega]_{A_p(\mathbb{R}^n)}:=\sup_{Q\subset\mathbb{R}^n}\left(\frac{1}{|Q|}\int_Q\omega(x)dx\right) \left(\frac{1}{|Q|}\int_Q\omega(x)^{\frac{1}{1-p}}dx\right)^{p-1}<\infty.$$
A weight $w$ is said to be of class $A_1(\mathbb{R}^n)$ if
$$[\omega]_{A_1(\mathbb{R}^n)}:=\sup_{Q\subset\mathbb{R}^n}\left(\frac{1}{|Q|}\int_Q\omega(x)dx \cdot\|\omega^{-1}\|_{L^{\infty}(\mathbb{R}^n)}\right)<\infty.$$
Define $A_{\infty}(\mathbb{R}^n):=\bigcup_{1\le p<\infty}A_p(\mathbb{R}^n)$. It is well-know that $A_p(\mathbb{R}^n)\subset A_q(\mathbb{R}^n)$ for $1\le p\le q\le\infty$.
\end{definition}
\begin{definition}
Let $0<p<\infty$ and $\omega\in A_{\infty}(\mathbb{R}^n)$. The weighted Lebesgue space $L_{\omega}^p(\mathbb{R}^n)$ is defined to be the set of all measurable functions $f$ on $\mathbb{R}^n$ such that
$$\|f\|_{L_{\omega}^p(\mathbb{R}^n)}:=\left(\int_{\mathbb{R}^n}|f(x)|^p\omega(x)dx\right)^{\frac{1}{p}}<\infty.$$
\end{definition}

Now, the definitions of mixed-norm Lebesgue spaces are given as follows, which was introduced by Benedek and Panzone \cite{7}.
\begin{definition} \label{def2.6}
Let $1<\vec{p}<\infty$. The mixed Lebesgue space $L^{\vec{p}}(\mathbb{R}^n)$ is defined by the set of all measurable
functions $f$ on $\mathbb{R}^n$, such that 
$$\left\|f\right\|_{L^{\vec{p}}(\mathbb{R}^n)}=\left(\int_{\mathbb{R}}\cdots\left(\int_{\mathbb{R}}\left|f(x)\right|^{p_1}\,dx_1\right)
^{\frac{p_2}{p_1}}\cdots\,dx_n\right)^{\frac{1}{p_n}}<\infty.$$
\end{definition}
\begin{remark}\label{re2.7}
\begin{itemize}
\item[(\romannumeral1)] Note that if $p_1=p_2=\cdots=p_n=p$, then $L^{\vec{p}}(\mathbb{R}^n)$ are reduced to the classical Lebesgue space $L^p(\mathbb{R}^n)$, and
    $$\left\|f\right\|_{L^{\vec{p}}(\mathbb{R}^n)}=\left\|f\right\|_{L^{p}(\mathbb{R}^n)}=\left(\int_{\mathbb{R}^n}\left|f(x)\right|^{p} dx\right)^{\frac{1}{p}}.$$
\item[(\romannumeral2)] The mixed-norm Lebesgue space $L^{\vec{p}}(\mathbb{R}^n)$ is a ball Banach function space (see \cite{TYYZ2021}).
\end{itemize}
\end{remark}
In 2019, Nogayama introduced mixed Morrey spaces \cite{15,N2019(2)} with combining mixed Lebesgue spaces and Morrey spaces as follows.
\begin{definition} \label{def2.8}
Let $0<\vec{p}\le\infty$ and $0<p_0\le\infty$ satisfy
$$\frac{n}{p_0}\le\sum_{j=1}^{n}\frac{1}{p_j}.$$
The mixed Morrey space $\mathcal{M}_{\vec{p}}^{p_0}(\mathbb{R}^n)$ is defined as the set of all measurable functions $f$ such that
$$\|f\|_{\mathcal{M}_{\vec{p}}^{p_0}(\R^n)}:=\sup\bigg\{|Q|^{\frac{1}{p_0} -\frac{1}{n}(\sum_{j=1}^n\frac{1}{p_j})}\|f\chi_Q\|_{L^{\vec{p}}(\mathbb{R}^n)}:Q~\text{is a cube in }\mathbb{R}^n\bigg\}<\infty.$$
\end{definition}
\begin{remark}\label{re2.9}
\begin{itemize}
\item[(\romannumeral1)]It is obvious that if $p_1=p_2=\cdots=p_n=p$, then $\mathcal{M}_{\vec{p}}^{p_0}(\R^n)=\mathcal{M}_p^{p_0}(\R^n)$ and if $\frac{1}{p_0}=\frac{1}{n}\sum_{j=1}^n\frac{1}{p_j}$, then $\mathcal{M}_{\vec{p}}^{p_0}(\R^n)=L^{\vec{p}}(\R^n).$
\item[(\romannumeral2)]The Mixed Morrey space $\mathcal{M}_{\vec{p}}^{p_0}(\mathbb{R}^n)$ is a ball Banach function space. In fact, according to \cite[Remark 3.1]{15}, mixed Morrey spaces are Banach spaces. Besides, it is easy to prove that:
    \begin{itemize}
    \item[(a)]If $|g|\le|f|$ almost everywhere, then
    $$\|g\chi_{Q}\|_{L^{\vec{p}(\mathbb{R}^n)}}\le\|f\chi_{Q}\|_{L^{\vec{p}(\mathbb{R}^n)}},$$
    $$\|g\chi_{Q}\|_{\mathcal{M}_{\vec{p}}^{p_0}(\mathbb{R}^n)}\le\|f\chi_{Q}\|_{\mathcal{M}_{\vec{p}}^{p_0}(\mathbb{R}^n)};$$
    \item[(b)]If $0\le f_m\uparrow f$ almost everywhere, then
    \begin{align*}
    \lim_{m\rightarrow\infty}\|f_m\|_{\mathcal{M}_{\vec{p}}^{p_0}(\mathbb{R}^n)}
    &=\sup_{Q\subset\mathbb{R}^n}|Q|^{\frac{1}{p_0} -\frac{1}{n}\sum_{j=1}^n\frac{1}{p_j}}\lim_{m\rightarrow\infty}\|f_m\chi_Q\|_{L^{\vec{p}}(\mathbb{R}^n)}\\
    &=\sup_{Q\subset\mathbb{R}^n}|Q|^{\frac{1}{p_0} -\frac{1}{n}\sum_{j=1}^n\frac{1}{p_j}}\|f\chi_Q\|_{L^{\vec{p}}(\mathbb{R}^n)}\\
    &=\|f\|_{\mathcal{M}_{\vec{p}}^{p_0}(\mathbb{R}^n)};
    \end{align*}
    \item[(c)]If $|Q|<\infty$, then
    $$\|\chi_Q\|_{\mathcal{M}_{\vec{p}}^{p_0}(\mathbb{R}^n)}=|Q|^{\frac{1}{p_0}}<\infty;$$
    \item[(d)]If $|Q|<\infty$ and $f>\infty$
     \begin{align*}
     \int_Q f(x)dx&\le |Q|^{\frac{1}{n}\sum_{i=1}^{n}\frac{1}{p'_i}}\cdot\|f\|_{L^{\vec{p}}(\mathbb{R}^n)}\\
     &=|Q|^{1-\frac{1}{p_0}}\cdot|Q|^{\frac{1}{p_0}-\frac{1}{n}\sum_{i=1}^{n}\frac{1}{p_i}}\|f\|
     _{L^{\vec{p}}(\mathbb{R}^n)}\\
     &\le|Q|^{1-\frac{1}{p_0}}\cdot\|f\|_{\mathcal{M}_{\vec{p}}^{p_0}(\mathbb{R}^n)}.
\end{align*}
But the authors pointed out that mixed Morrey spaces are not Banach function spaces (see \cite[Example 3.3]{18}).
\end{itemize}
\end{itemize}
\end{remark}

In \cite{N2019(2)}, Nogayama introduced the block space $\mathcal{B}_{\vec{p}\,'}^{p'_0}(\mathbb{R}^n)$, which is the predual space of mixed Morrey space.
\begin{definition} \label{def2.10}
Let $1<p_0,\vec{p}<\infty$ and $\frac{n}{p_0}<\sum_{i=1}^{n}\frac{1}{p_i}$. A measurable function $b(x)$ is said to be a $(p'_0,\vec{p}\,')$-block if there exists a cube $Q$ such that
$$\mathrm{supp}\,b\subset Q,~\|b\|_{L^{\vec{p}\,'}(\mathbb{R}^n)}\le |Q|^{\frac{1}{p_0}-\frac{1}{n}\sum_{i=1}^{n}\frac{1}{p_i}}.$$
The block space $\mathcal{B}_{\vec{p}\,'}^{p'_0}(\mathbb{R}^n)$ is denoted by that a measurable function set of $f=\sum_{i=1}^{\infty}\lambda_{i}b_i(x)$, where $\{\lambda_{i}\}_{i=1}^{\infty}\in \ell^{1}$ and $b_{i}$ is a $(p'_0,\vec{p}\,')$-block for each $i$. The norm $\|f\|_{\mathcal{B}_{\vec{p}\,'}^{p'_0}(\mathbb{R}^n)}$ for $f\in \mathcal{B}_{\vec{p}\,'}^{p'_0}(\mathbb{R}^n)$ is defined as
\begin{align*}
\|f\|_{\mathcal{B}_{\vec{p}\,'}^{p'_0}(\mathbb{R}^n)} =\inf\{\|\{\lambda_{i}\}_{i=1}^{\infty}\|_{\ell^{1}}:f&=\sum_{i=1}^{\infty}\lambda_{i}b_i(x),\\
&\{\lambda_{i}\}_{i=1}^{\infty}\in \ell^{1},b_{i}~\text{is a}~(p'_0,\vec{p}\,')-\text{block for any }i\}.
\end{align*}
\end{definition}

Let us recall the definition of the $BMO(\mathbb{R}^n)$ space.
\begin{definition} \label{def2.12}
If $b$ is a measurable function on $\mathbb{R}^n$ and satisfies that
$$\|b\|_{BMO(\mathbb{R}^n)}=\sup_{Q\subset\mathbb{R}^n}\frac{1}{|Q|}\int_{Q}|b(y)-b_Q|dy<\infty,$$
then $b\in BMO(\mathbb{R}^n)$ and $\|b\|_{BMO(\mathbb{R}^n)}$ is the norm of $b$ in $BMO(\mathbb{R}^n)$.
\end{definition}

\section{Weak convergence of mixed-norm Lebesgue spaces}\label{sec3}
\par
In this section, we discuss the separable and weak convergence of mixed-norm Lebesgue spaces.
\begin{theorem}\label{Seth}
Let $1<\vec{p}<\infty$. Then the mixed-norm Lebesgue space $L^{\vec{p}}(\mathbb{R}^n)$ is separable space.
\end{theorem}
\begin{theorem}\label{WCth}
Let $1<\vec{p}<\infty$ and $\{f_k\}_{k=1}^{\infty}\subset L^{\vec{p}}(\mathbb{R}^n)$. If there exists a positive constant $M$ such that
$$\|f_k\|_{L^{\vec{p}}(\mathbb{R}^n)}<M,$$
then there exists a subset $\{f_{k_j}\}_{j=1}^{\infty}$ is weak convergence on $L^{\vec{p}}(\mathbb{R}^n)$.
\end{theorem}

Before we give our proofs, the following lemma is necessary. According to \cite[Proposition 3.8]{TYYZ2021}, the following lemma can be obtained.
\begin{lemma}\label{le3.2}
Let $1<\vec{p}<\infty$. Then $C_c(\mathbb{R}^n)$ is dense in $L^{\vec{p}}(\mathbb{R}^n)$, where $C_c(\mathbb{R}^n)$ denoted by the set of continuous functions with compact support.
\end{lemma}

We denote by $\mathcal{Q}_k(k\in \mathbb{Z})$ the collection of cubes in $\mathbb{R}^n$ which are congruent to $[0,2^{-k})^n$ and vertices lie on the lattice $2^{-k}\mathbb{Z}^n$, that is, $\mathcal{Q}_k=\{2^{-k}(i+[0,1)^n):i\in \mathbb{Z}^n\}(k\in \mathbb{Z})$. The cubes in $\mathcal{D}=\bigcup_{k\in \mathbb{Z}}\mathcal{Q}_k$ are called dyadic cubes. Now, the proofs of Theorem \ref{Seth} and Theorem \ref{WCth} can be given.

\begin{proof}[Proof of Theorem \ref{Seth}]
From Lemma \ref{le3.2}, for any $f\in L^{\vec{p}}(\mathbb{R}^n)$ and $\varepsilon<\infty$, there exist $g\in C_c(\mathbb{R}^n)$ such that
$$\|f-g\|_{L^{\vec{p}}(\mathbb{R}^n)}<\varepsilon.$$
It is easy to know that $g$ is uniformly continuous. Hence, there exists a sequence of dyadic cube $\{Q_i\}_{i=1}^{N}$ and a sequence of rational number $\{c_i\}_{i=1}^{N}$ such that
$$\|g-\sum_{i=1}^{N}c_i\chi_{Q_i}\|_{L^{\vec{p}}(\mathbb{R}^n)}<\varepsilon.$$
We write $\Gamma$ as a set of simple functions $\varphi$ and
$$\varphi(x)=\sum_{i=1}^{N}c_i\chi_{Q_i},$$
where $\{Q_i\}_{i=1}^{N}$ is a sequence of dyadic cube and $\{c_i\}_{i=1}^{N}$ is a sequence of rational numbers. It is obvious that $\Gamma$ is countable and dense in $L^{\vec{p}}(\mathbb{R}^n)$. Thus, $L^{\vec{p}}(\mathbb{R}^n)$  is separable.
\end{proof}

\begin{proof}[Proof of Theorem \ref{WCth}]
According to \cite[Theorem 1.a]{7}, we know that the dual of $L^{\vec{p}}(\mathbb{R}^n)$ is $L^{\vec{p}\,'}(\mathbb{R}^n)$, where $\frac{1}{\vec{p}}+\frac{1}{\vec{p}\,'}=1$ and $1<\vec{p}<\infty$. Hence we only need to prove that there exists a subset $\{f_{k_j}\}_{j=1}^{\infty}$ such that for any $g\in L^{\vec{p}\,'}(\mathbb{R}^n)$,
$$\lim_{j\rightarrow\infty}\int_{\mathbb{R}^n}f_{k_j}(x)g(x)dx=\int_{\mathbb{R}^n}f(x)g(x)dx,$$
where $f\in L^{\vec{p}}(\mathbb{R}^n)$.

According to Theorem \ref{Seth}, we assume that $\{g_i\}_{i=1}^{\infty}$ is dense in $L^{\vec{p}\,'}(\mathbb{R}^n)$. Write
$$\mathcal{F}_k(g)=\int_{\mathbb{R}^n}f_{k}(x)g(x)dx.$$
Using H\"older's inequality,
$$|\mathcal{F}_k(g_i)|\le M\|g_i\|_{L^{\vec{p}\,'}(\mathbb{R}^n)}.$$
According to the boundedness of $\{\mathcal{F}_k(g_1)\}_{k=1}^{\infty}$, there exist convergent subsequences $\{\mathcal{F}_{k,1}(g_1)\}_{k=1}^{\infty}$. By the same argument, we can find a subsequence $\{\mathcal{F}_{k,2}(g_2)\}_{k=1}^{\infty}$ from $\{\mathcal{F}_{k,1}(g_2)\}_{k=1}^{\infty}$ such that $\{\mathcal{F}_{k,2}(g_2)\}_{k=1}^{\infty}$ is convergence. So for any $g_{i_0}( i_0\le j)$ there exists a convergent subsequences $\{\mathcal{F}_{k,j}(g_{i_0})\}$. By a diagonal process one can obtain that a subsequence $\{\mathcal{F}_{j,j}(g_{i})\}_{j=1}^{\infty}$ is convergence for any $g_i$ and
\begin{align*}
\mathcal{F}_{j,j}(g)&=\int_{\mathbb{R}^n}f_{k,j}(x)g(x)dx
:=\int_{\mathbb{R}^n}f_{k_j}(x)g(x)dx.
\end{align*}
For any $g\in L^{\vec{p}\,'}(\mathbb{R}^n)$ and any $\varepsilon>0$, there exists $g_i$ such that
$$\|g-g_i\|_{L^{\vec{p}\,'}(\mathbb{R}^n)}\le\varepsilon/2M.$$
Hence
\begin{align*}
|\mathcal{F}_{m,m}(g)-\mathcal{F}_{m',m'}(g)| &\le\int_{\mathbb{R}^n}|f_{k_m}(x)-f_{k_{m'}}(x)||g_i(x)|dx\\
&+\int_{\mathbb{R}^n}|f_{k_m}(x)-f_{k_{m'}}(x)||(g(x)-g_i(x)|dx\\
&\le \int_{\mathbb{R}^n}|f_{k_m}(x)-f_{k_{m'}}(x)||g_i(x)|dx+\varepsilon.
\end{align*}
When $m$ and $m'$ are large enough,
$$|\mathcal{F}_{m,m}(g)-\mathcal{F}_{m',m'}(g)|\le 2\varepsilon.$$
Thus $\{\mathcal{F}_{j,j}(g)\}_{j=1}^{\infty}$ is a Cauchy sequence for any $g\in L^{\vec{p}\,'}(\mathbb{R}^n)$. Then let
$$\mathcal{F}(g)=\lim_{j\rightarrow\infty}\mathcal{F}_{j,j}(g)$$
and $\mathcal{F}(g)$ be a linear bounded functional on $L^{\vec{p}\,'}(\mathbb{R}^n)$. Applying \cite[Theorem 1.a]{7}, we see that there exist $f\in L^{\vec{p}}(\mathbb{R}^n)$ such that
$$\int_{\mathbb{R}^n}f(x)g(x)dx=\mathcal{F}(g)=\lim_{j\rightarrow\infty}\mathcal{F}_{j,j}(g) =\lim_{j\rightarrow\infty}\int_{\mathbb{R}^n}f_{k_j}(x)g(x)dx.$$
The proof is completed.
\end{proof}

\section{The K\"othe dual spaces of mixed Morrey spaces}\label{sec4}
In this section, we study the Fatou property of the block spaces and prove that the block spaces are the K\"othe dual spaces of mixed Morrey spaces.
\begin{theorem}\label{Fath}
Let $1<p_0,\vec{p}<\infty$, $\frac{n}{p_0}<\sum_{i=1}^{n}\frac{1}{p_i}$. If $f_k\in \mathcal{B}_{\vec{p}\,'}^{p'_0}(\mathbb{R}^n)(k\in \mathbb{N})$ are nonnegative functions, $\|f_k\|_{\mathcal{B}_{\vec{p}\,'}^{p'_0}(\mathbb{R}^n)}\le 1$ and $f_k\uparrow f~a.e.$, then
$$\lim_{k\rightarrow\infty}\|f_k\|_{\mathcal{B}_{\vec{p}\,'}^{p'_0}(\mathbb{R}^n)}=\|f\|_{\mathcal{B}_{\vec{p}\,'}^{p'_0}(\mathbb{R}^n)}.$$
\end{theorem}
\begin{theorem}\label{KDth}
Let $1<p_0,\vec{p}<\infty$, $\frac{n}{p_0}<\sum_{i=1}^{n}\frac{1}{p_i}$. $\mathcal{B}_{\vec{p}\,'}^{p'_0}(\mathbb{R}^n)$ is the K\"othe dual of $\mathcal{M}_{\vec{p}}^{p_0}(\mathbb{R}^n)$.
\end{theorem}

In fact, a abstract result can be found in \cite[Theorem 2.2]{MST2016} for Theorem \ref{Fath}. To state our proofs better, we prove Theorem \ref{Fath} by the method from \cite{18} instead of proving the conditions in \cite[Theorem 2.2]{MST2016}. We need Lemma \ref{le4.3} and Lemma \ref{Dth} (see \cite[Theorem 2.7]{N2019(2)}) before our proof.
\begin{lemma}\label{le4.3}
Let $1<p_0,\vec{p}<\infty$, $\frac{n}{p_0}<\sum_{i=1}^{n}\frac{1}{p_i}$ and $f\in \mathcal{B}_{\vec{p}\,'}^{p'_0}(\mathbb{R}^n)$. Then $f$ can be decomposed as
$$f(x)=\sum_{Q\in\mathcal{D}}\lambda_{Q}b_{Q}(x),$$
where $\lambda_{Q}$ is a nonnegative number with
$$\sum_{Q\in\mathcal{D}}\lambda_{Q}\le 2\|f\|_{\mathcal{B}_{\vec{p}\,'}^{p'_0}(\mathbb{R}^n)}\cdot 3^n,$$
and $b_{Q}$ is a $(p'_0,\vec{p}\,')$-block with $\mathrm{supp}\,b_{Q}\subset 3Q$.
\end{lemma}
\begin{lemma}\label{Dth}
Let $1<p_0,\vec{p}<\infty$ and $\frac{n}{p_0}<\sum_{i=1}^{n}\frac{1}{p_i}$. The block space $\mathcal{B}_{\vec{p}\,'}^{p'_0}(\mathbb{R}^n)$ is the predual spaces of mixed Morrey spaces. Furthermore, for any $f\in\mathcal{M}_{\vec{p}}^{p_0}(\mathbb{R}^n)$ and $g\in\mathcal{B}_{\vec{p}\,'}^{p'_0}(\mathbb{R}^n)$,
$$\|f\|_{\mathcal{M}_{\vec{p}}^{p_0}(\mathbb{R}^n)}=\sup\left\{\int_{\mathbb{R}^n}\left|f(x)h(x)\right|dx:h\in \mathcal{B}_{\vec{p}\,'}^{p'_0}(\mathbb{R}^n),\|h\|_{\mathcal{B}_{\vec{p}\,'}^{p'_0}(\mathbb{R}^n)}=1\right\}$$
and
$$\|g\|_{\mathcal{B}_{\vec{p}\,'}^{p'_0}(\mathbb{R}^n)}=\max\left\{\int_{\mathbb{R}^n}\left|h(x)g(x)\right|dx:h\in \mathcal{M}_{\vec{p}}^{p_0}(\mathbb{R}^n),\|h\|_{\mathcal{B}_{\vec{p}}^{p_0}(\mathbb{R}^n)}=1\right\}.$$
\end{lemma}

Now, let us prove Lemma \ref{le4.3}, Theorem \ref{Fath} and Theorem \ref{KDth}.
\begin{proof}[Proof of Lemma \ref{le4.3}]
For $f\in \mathcal{B}_{\vec{p}\,'}^{p'_0}(\mathbb{R}^n)$, there exist $\{\lambda_i\}_{i=1}^{\infty}$ and $\{b_i\}_{i=1}^{\infty}$ such that
$$f(x)=\sum_{i=1}^{\infty}\lambda_ib_i(x),$$
where $\sum_{i=1}^{\infty}|\lambda_i|\le 2\|f\|_{\mathcal{B}_{\vec{p}\,'}^{p'_0}(\mathbb{R}^n)}$ and $b_i$ is a $(p'_0,\vec{p}\,')$-block with $\text{supp}~b_i\subset Q'_i$. We divide $\mathbb{N}$ into the disjoint sets $K(Q)$, $Q\in\mathcal{D}$, as
$$\mathbb{N}=\bigcup_{Q\in \mathcal{D}}K(Q),$$
and if $i\in K(Q)$ then
$$\text{supp}~b_i\subset 3Q~\text{and }|Q'_i|\ge |Q|.$$
Then
\begin{align*}
f(x)&=\sum_{i=1}^{\infty}\lambda_ib_i(x)=\sum_{Q\in \mathcal{D}}\left(\sum_{i\in K(Q)}\lambda_ib_i(x)\right)\\
&=\sum_{Q\in \mathcal{D}}\left[3^n\sum_{i\in K(Q)}|\lambda_i|\right]\cdot\left[\left(3^n\sum_{i\in K(Q)}|\lambda_i|\right)^{-1}\sum_{i\in K(Q)}\lambda_ib_i(x)\right]\\
&=:\sum_{Q\in \mathcal{D}}\lambda_{Q}b_Q(x).
\end{align*}
It is easy to prove that
\begin{align*}
\sum_{Q\in \mathcal{D}}\lambda_{Q}=3^n\sum_{Q\in \mathcal{D}}\sum_{i\in K(Q)}|\lambda_i|
\le 3^n\sum_{i=1}^{\infty}|\lambda_i|\le  2\|f\|_{\mathcal{B}_{\vec{p}\,'}^{p'_0}(\mathbb{R}^n)}\cdot 3^n,
\end{align*}
and
\begin{align*}
\|b_Q(x)\|_{L^{\vec{p}\,'}(\mathbb{R}^n)} &=\left(3^n\sum_{i\in K(Q)}|\lambda_i|\right)^{-1}\|\sum_{i\in K(Q)}\lambda_ib_i(x)\|_{L^{\vec{p}\,'}(\mathbb{R}^n)}\\
&\le\left(3^n\sum_{i\in K(Q)}|\lambda_i|\right)^{-1}\sum_{i\in K(Q)}|\lambda_i|\|b_i(x)\|_{L^{\vec{p}\,'}(\mathbb{R}^n)}\\
&\le\left(3^n\sum_{i\in K(Q)}|\lambda_i|\right)^{-1}\sum_{i\in K(Q)}|\lambda_i||Q'_i|^{\frac{1}{p_0}-\frac{1}{n}\sum_{i=1}^{n}\frac{1}{p_i}}\\
&\le |Q|^{\frac{1}{p_0}-\frac{1}{n}\sum_{i=1}^{n}\frac{1}{p_i}}\left(3^n\sum_{i\in K(Q)}|\lambda_i|\right)^{-1}\sum_{i\in K(Q)}|\lambda_i|\\
&= |3Q|^{\frac{1}{p_0}-\frac{1}{n}\sum_{i=1}^{n}\frac{1}{p_i}}\cdot 3^{n(\frac{1}{n}\sum_{i=1}^{n}\frac{1}{p_i}-\frac{1}{p_0}-1)}\\
&\le |3Q|^{\frac{1}{p_0}-\frac{1}{n}\sum_{i=1}^{n}\frac{1}{p_i}}.
\end{align*}
Hence $b_Q$ is a $(p'_0,\vec{p}\,')$-block with $\text{supp}~b_{Q}\subset 3Q$. The proof is completed.
\end{proof}
\begin{proof}[Proof of Theorem \ref{Fath}]
For any nonnegative function $f_k$, Lemma \ref{le4.3} yields
$$f_k(x)=\sum_{}\lambda_{Q,k}b_{Q,k}(x),$$
where $\lambda_{Q,k}$ is nonnegative number with
\begin{equation}\label{eq4.1}
\sum_{Q\in \mathcal{D}}\lambda_{Q,k}\le 2\cdot 3^n,
\end{equation}
and $b_{Q,k}$ is a $(p_0,\vec{p})$-block with
$$\text{supp}~b_{Q,k}\subset 3Q,~\|b_{Q,k}\|_{L^{\vec{p}\,'}(\mathbb{R}^n)}\le |3Q|^{\frac{1}{p_0}-\frac{1}{n}\sum_{i=1}^{n}\frac{1}{p_i}}.$$
In view of Theorem \ref{WCth} and diagonalization argument, for any $Q\in\mathcal{D}$, one can obtain that there exist subsequences $\{\lambda_{Q,k_j}\}_{j=1}^{\infty}$ and $\{b_{Q,k_j}\}_{j=1}^{\infty}$ such that
\begin{equation}\label{eq4.2}
\lim_{j\rightarrow \infty}\int_{\mathbb{R}^n}b_{Q,k_j}(x)g(x)dx=\int_{\mathbb{R}^n}b_{Q}(x)g(x)dx,\text{ for any }g\in L^{\vec{p}}(3Q),
\end{equation}
\begin{equation}\label{eq4.3}
\lim_{j\rightarrow \infty}\lambda_{Q,k_j}=\lambda_Q,
\end{equation}
\begin{equation}\label{eq4.4}
f_{k_j}(x)=\sum_{Q\in\mathcal{D}}\lambda_{Q,k_j}b_{Q,k_j}(x).
\end{equation}
It is obvious that $\text{supp}~b_{Q}\subset 3Q.$ Furthermore, by \cite[Theorem 1]{7}, there exist $g\in L^{\vec{p}}(3Q)$ and $\|g\|_{L^{\vec{p}}(3Q)}\le 1$ such that
$$\|b_{Q}(x)\|_{L^{\vec{p}\,'}(\mathbb{R}^n)}=\int_{\mathbb{R}^n}|b_Q(x)g(x)|dx.$$
Hence
\begin{align*}
\|b_{Q}\|_{L^{\vec{p}\,'}(\mathbb{R}^n)}&=\int_{\mathbb{R}^n}|b_Q(x)g(x)|dx\\
&\le \liminf_{j\rightarrow\infty}\int_{\mathbb{R}^n}|b_{Q,k_j}(x)g(x)|dx\\
&\le\|g\|_{L^{\vec{p}}(3Q)}\liminf_{j\rightarrow\infty}\|b_{Q,k_j}\|_{L^{\vec{p}\,'}(3Q)}\\
&\le |3Q|^{\frac{1}{p_0}-\frac{1}{n}\sum_{i=1}^{n}\frac{1}{p_i}}.
\end{align*}
Thus $b_{Q}$ is a $(p'_0,\vec{p}\,')$-block. Moreover, we conclude that from the Fatou theorem,
\begin{equation}\label{eq4.5}
\sum_{Q\in \mathcal{D}}\lambda_{Q}\le \liminf_{j\rightarrow\infty}\sum_{Q\in \mathcal{D}}\lambda_{Q,k_j}\le 2\cdot 3^n.
\end{equation}
Then,
$$f_0(x):=\sum_{Q\in \mathcal{D}}\lambda_{Q}b_{Q}(x)\in\mathcal{B}_{\vec{p}\,'}^{p'_0}(\mathbb{R}^n).$$

Next, let us prove that $f(x)=f_0(x)~a.e.$. By the Lebesgue differential theorem, we only need to prove that
$$\lim_{j\rightarrow\infty}\int_{Q_0}f_{k_j}(x)dx=\int_{Q_0}f(x)dx=\int_{Q_0}f_0(x)dx$$
holds for any $Q_0\in\mathcal{D}$. Without loss of generality, let $|Q_0|\le 1$ and $0<\varepsilon<1$. Let
\begin{displaymath}
\left\{ \begin{array}{ll}
\mathcal{D}_1(Q_0):=\{Q\in\mathcal{D}:3Q\cap Q_0\neq\emptyset, |3Q|\le c_1\},\\
\mathcal{D}_2(Q_0):=\{Q\in\mathcal{D}:3Q\cap Q_0\neq\emptyset, |3Q|\in(c_1,c_2)\},\\
\mathcal{D}_3(Q_0):=\{Q\in\mathcal{D}:3Q\cap Q_0\neq\emptyset, |3Q|\ge c_2\},
\end{array} \right.
\end{displaymath}
where
$$c_1^{\frac{1}{p_0}}=\frac{\varepsilon}{12\cdot 3^n},~~
c_2^{\frac{1}{p_0}-\frac{1}{n}\sum_{i=1}^{\infty}\frac{1}{p_i}}=\frac{\varepsilon}{12\cdot 3^n}.$$
For the set $\mathcal{D}_1(Q_0)$, by (\ref{eq4.1}), (\ref{eq4.5}) and H\"older's inequality
\begin{align*}
&\quad\sum_{Q\in\mathcal{D}_1(Q_0)}\int_{Q_0}|\lambda_{Q,k_j}b_{Q,k_j}(x)-\lambda_{Q}b_{Q}(x)|dx \\ &\le\sum_{Q\in\mathcal{D}_1(Q_0)}\left(\lambda_{Q,k_j}\int_{3Q}|b_{Q,k_j}(x)dx+\lambda_{Q}\int_{3Q}|b_{Q}(x)|dx\right)\\
&\le\sum_{Q\in\mathcal{D}_1(Q_0)}\left(\lambda_{Q,k_j}+\lambda_{Q}\right)|3Q|^{\frac{1}{p_0}}\\
&\le 2\cdot 2\cdot 3^nc_1^{\frac{1}{p_0}}
\le \frac{\varepsilon}{3}.
\end{align*}
For the set $\mathcal{D}_3(Q_0)$, from (\ref{eq4.1}), (\ref{eq4.5}) and H\"older's inequality on can deduce that
\begin{align*}
&\quad\sum_{Q\in\mathcal{D}_3(Q_0)}\int_{Q_0}|\lambda_{Q,k_j}b_{Q,k_j}(x)-\lambda_{Q}b_{Q}(x)|dx \\ &\le\sum_{Q\in\mathcal{D}_3(Q_0)}\left(\lambda_{Q,k_j}\int_{Q_0}|b_{Q,k_j}(x)dx+\lambda_{Q}\int_{Q_0}|b_{Q}(x)|dx\right)\\
&\le|Q_0|^{\frac{1}{n}\sum_{i=1}^{\infty}\frac{1}{p_i}} \sum_{Q\in\mathcal{D}_3(Q_0)}\left(\lambda_{Q,k_j}+\lambda_{Q}\right)|3Q|^{\frac{1}{p_0}-\frac{1}{n}\sum_{i=1}^
{\infty}\frac{1}{p_i}}\\
&\le 2\cdot 2\cdot 3^nc_2^{\frac{1}{p_0}-\frac{1}{n}\sum_{i=1}^{\infty}\frac{1}{p_i}}
\le \frac{\varepsilon}{3}.
\end{align*}
For the set $\mathcal{D}_2(Q_0)$, by $|Q_0|\in (c_1,c_2)$, we know that $\mathcal{D}_2(Q_0)$ contains the only finite number of dyadic cubes. When $j$ large enough, applying (\ref{eq4.2}) and (\ref{eq4.3}) we see,
\begin{align*}
&\quad\sum_{Q\in\mathcal{D}_2(Q_0)}\left|\int_{Q_0}\lambda_{Q,k_j}b_{Q,k_j}(x)-\lambda_{Q}b_{Q}(x)dx\right|\\
&\le \sum_{Q\in\mathcal{D}_2(Q_0)}|\lambda_{Q,k_j}|\left|\int_{Q_0}b_{Q,k_j}(x)-b_{Q}(x)dx\right|\\
&\quad+\sum_{Q\in\mathcal{D}_2(Q_0)}|\lambda_{Q,k_j}-\lambda_{Q}|\left|\int_{Q_0}b_{Q}(x)-b_{Q}(x)dx\right|
\le \frac{\varepsilon}{3}.
\end{align*}
Hence, $f_0(x)=f(x)$.

To finish the proof, we need to know that the expressions
$$\lim_{k\rightarrow\infty}\|f_k\|_{\mathcal{B}_{\vec{p}\,'}^{p'_0}(\mathbb{R}^n)}=\|f\|_{\mathcal{B}_{\vec{p}\,'}^{p'_0}(\mathbb{R}^n)}.$$
Using Lemma \ref{Dth} and the dominated convergence theorem, we show that there exist $g\in\mathcal{M}_{\vec{p}}^{p_0}(\mathbb{R}^n)$ such that
\begin{align*}
\|f\|_{\mathcal{B}_{\vec{p}\,'}^{p'_0}(\mathbb{R}^n)}&=\sup_{\|g\|_{\mathcal{M}_{\vec{p}}^{p_0}(\mathbb{R}^n)}\le 1} \int_{\mathbb{R}^n}\lim_{k\rightarrow\infty}\left|f_k(x)g(x)\right|dx\\
&=\lim_{k\rightarrow\infty}\sup_{\|g\|_{\mathcal{M}_{\vec{p}}^{p_0}(\mathbb{R}^n)}\le 1} \int_{\mathbb{R}^n}\left|f_k(x)g(x)\right|dx\\
&=\lim_{k\rightarrow\infty}\|f_k\|_{\mathcal{B}_{\vec{p}\,'}^{p'_0}(\mathbb{R}^n)}.
\end{align*}
The proof is completed.
\end{proof}

\begin{proof}[Proof of Theorem \ref{KDth}]
By Lemma \ref{Dth},
$$\mathcal{B}_{\vec{p}\,'}^{p'_0}(\mathbb{R}^n)\subset(\mathcal{M}_{\vec{p}}^{p_0}(\mathbb{R}^n))'.$$
It suffices to show that if $f$ satisfies that
$$\sup\left\{\int_{\mathbb{R}^n}\left|f(x)g(x)\right|dx:g\in \mathcal{M}_{\vec{p}}^{p_0}(\mathbb{R}^n),\|g\|_{\mathcal{M}_{\vec{p}}^{p_0}(\mathbb{R}^n)}\le 1\right\}=M<\infty,$$
then $f\in\mathcal{B}_{\vec{p}\,'}^{p'_0}(\mathbb{R}^n)$.

Without loss of generality, let $f\ge0$. For $k=1,2,\cdots$, set $Q_k=(-k,k)^n$, let
$$f_k(x):=\min\{f(x)/M,k/M\}\chi_{Q_k}(x).$$
Notice that
$$\|f_k\|_{L^{\vec{p}\,'}(\mathbb{R}^n)}\le \frac{k}{M}|Q_k|^{1-\frac{1}{n}\sum_{i=1}^{n}\frac{1}{p_i}}\le \frac{k}{M}|Q_k|^{1-\frac{1}{p_0}}\cdot|Q_k|^{\frac{1}{p_0}-\frac{1}{n}\sum_{i=1}^{n}\frac{1}{p_i}}.$$
Hence $f_k\in\mathcal{B}_{\vec{p}\,'}^{p'_0}(\mathbb{R}^n)$. Lemma \ref{Dth} yields
\begin{align*}
\|f_k\|_{\mathcal{B}_{\vec{p}\,'}^{p'_0}(\mathbb{R}^n)}&=\sup\left\{\int_{\mathbb{R}^n}\left|f_k(x)g(x)\right|dx:g\in \mathcal{M}_{\vec{p}}^{p_0}(\mathbb{R}^n),\|g\|_{\mathcal{M}_{\vec{p}}^{p_0}(\mathbb{R}^n)}\le 1\right\}\\
&\le \frac{1}{M}\sup\left\{\int_{\mathbb{R}^n}\left|f(x)g(x)\right|dx:g\in \mathcal{M}_{\vec{p}}^{p_0}(\mathbb{R}^n),\|g\|_{\mathcal{M}_{\vec{p}}^{p_0}(\mathbb{R}^n)}\le 1\right\}
\le 1.
\end{align*}
By the facts $f_k\uparrow f/M~a.e.$ and Theorem \ref{Fath}, we conclude that  $f\in\mathcal{B}_{\vec{p}\,'}^{p'_0}(\mathbb{R}^n)$.
\end{proof}

\section{The boundedness of the Hardy--Littlewood maximal function on the block spaces $\mathcal{B}_{\vec{p}\,'}^{p'_0}(\mathbb{R}^n)$}\label{sec5}
\par
In this section, we prove the boundedness of the Hardy--Littlewood maximal function on the block space $\mathcal{B}_{\vec{p}\,'}^{p'_0}(\mathbb{R}^n)$.
\begin{theorem}\label{th5.1}
Let $1<p_0,\vec{p}<\infty$ and $\frac{n}{p_0}<\sum_{i=1}^{n}\frac{1}{p_i}$. $M$ is bounded on  $\mathcal{B}_{\vec{p}\,'}^{p'_0}(\mathbb{R}^n)$.
\end{theorem}

Before our proof, let us recall some necessary lemmas. The Lemma \ref{le5.2} and Lemma \ref{Mth} can be found in \cite[Lemma 3.5]{27} and \cite[Lemma 2.2]{28}.
\begin{lemma}\label{le5.2}
Let $1<\vec{p}\le\infty$. Then there exists a positive constant $C$, depending on $\vec{p}$, such that, for any $f\in L^{\vec{p}}(\mathbb{R}^n)$,
$$\|Mf\|_{L^{\vec{p}}(\mathbb{R}^n)}\le C\|f\|_{L^{\vec{p}}(\mathbb{R}^n)}.$$
\end{lemma}

\begin{lemma}\label{Mth}
If $\mathbf{X}$ is a ball Banach function space and $M$ is bounded on $\mathbf{X}$, then
$$|Q|\sim\|\chi_{Q}\|_{\mathbf{X}}\|\chi_Q\|_{\mathbf{X}'}.$$
\end{lemma}

By a similar argument to \cite{ST2009}, we show the detailed proof of Theorem \ref{th5.1} as follows.
\begin{proof}[Proof of Theorem \ref{th5.1}]
Suppose that $b$ is a $(p_0,\vec{p})$-block and $\text{supp}~b\subset Q(x_0,r)$. Let $Q_k=Q(x_0,2^kr)$, $m_k(x)=\chi_{Q_{k+1}\setminus Q_{k}}(x)M(b)(x)$ and $m_0(x)=\chi_{Q_{1}}(x)M(b)(x)$, for $k=1,2,...$. Then
$$Mb(x)=\sum_{k=0}^{\infty}m_k(x),$$
and for $k=0,1,2,\cdots$,
$$\text{supp}~m_k\subset Q_{k+1}.$$
Applying Lemma \ref{le5.2}, we deduce that
$$\|m_0\|_{L^{\vec{p}\,'}(\mathbb{R}^n)}\le \|Mb\|_{L^{\vec{p}\,'}(\mathbb{R}^n)}\lesssim\|b\|_{L^{\vec{p}\,'}(\mathbb{R}^n)} \le2^{\frac{1}{n}\sum_{i=1}^{n}\frac{1}{p_i}-\frac{1}{p_0}}|Q_1| ^{\frac{1}{p_0}-\frac{1}{n}\sum_{i=1}^{n}\frac{1}{p_i}}.$$
And using the definition of $M$ and H\"older's inequality, for any $k\in\mathbb{N}\backslash\{0\}$,
\begin{align*}
|m_k(x)|&=\chi_{Q_{k+1}\backslash Q_k}(x)|Mb(x)|\\
&\lesssim \frac{\chi_{Q_{k+1}\backslash Q_k}(x)}{(2^kr)^n}\int_{Q(x_0,r)}b(y)dy\\
&\le \frac{\chi_{Q_{k+1}\backslash Q_k}(x)}{(2^kr)^n}\|b\|_{L^{\vec{p}\,'}(\mathbb{R}^n)}\|\chi_{Q(x_0,r)}\|_{L^{\vec{p}}(\mathbb{R}^n)}.
\end{align*}
Applying Lemma \ref{Mth}, we write
\begin{align*}
\|m_k\|_{L^{\vec{p}\,'}(\mathbb{R}^n)}&\lesssim \frac{\|\chi_{Q_{k+1}}\|_{L^{\vec{p}\,'}(\mathbb{R}^n)}}{(2^kr)^n}\|b\|_{L^{\vec{p}\,'}(\mathbb{R}^n)}\|\chi_{Q(x_0,r)}\|_{L^{\vec{p}}(\mathbb{R}^n)}\\
&=\frac{\|\chi_{Q_{k+1}}\|_{L^{\vec{p}\,'}(\mathbb{R}^n)}\cdot\|\chi_{Q_{k+1}}\|_{L^{\vec{p}}(\mathbb{R}^n)}} {\|\chi_{Q_{k+1}}\|_{L^{\vec{p}}(\mathbb{R}^n)}}\cdot \frac{\|b\|_{L^{\vec{p}\,'}(\mathbb{R}^n)}\|\chi_{Q(x_0,r)}\|_{L^{\vec{p}}(\mathbb{R}^n)}}{(2^kr)^n}\\
&\sim \frac{|Q(x_0,2^{k+1}r)|} {\|\chi_{Q_{k+1}}\|_{L^{\vec{p}}(\mathbb{R}^n)}}\cdot \frac{\|b\|_{L^{\vec{p}\,'}(\mathbb{R}^n)}\|\chi_{Q(x_0,r)}\|_{L^{\vec{p}}(\mathbb{R}^n)}}{(2^kr)^n}\\
&\lesssim \left(\frac{1}{2^{nk}}\right)^{\frac{1}{p_0}}\cdot |Q_{k+1}|^{\frac{1}{p_0}-\frac{1}{n}\sum_{i=1}^{n}\frac{1}{p_i}}.
\end{align*}
It is obvious that $\frac{1}{p_0}>0$ and
$$\sum_{k=1}^{\infty}\left(\frac{1}{2^{nk}}\right)^{\frac{1}{p_0}}<C_{n,p_0,\vec{p}}<\infty,$$
where $C_{n,p_0,\vec{p}}$ only depends on $n,~p_0$ and $\vec{p}$. Then, for any $(p_0',\vec{p}\,')$-block $b$,
$$\|Mb\|_{\mathcal{B}_{\vec{p}\,'}^{p'_0}(\mathbb{R}^n)}\le C_{n,p_0,\vec{p}}.$$

If $f\in \mathcal{B}_{\vec{p}\,'}^{p'_0}(\mathbb{R}^n)$, then there exists a decomposition such that $f=\sum_{k=1}^{\infty}\lambda_{k}b_{k}(x)$ and
$$\sum_{k=1}^{\infty}|\lambda_{k}|\le 2\|f\|_{\mathcal{B}_{\vec{p}\,'}^{p'_0}(\mathbb{R}^n)}.$$
Hence
\begin{align*}
\|Mf\|_{\mathcal{B}_{\vec{p}\,'}^{p'_0}(\mathbb{R}^n)} \le\sum_{k=1}^{\infty}|\lambda_k|\|Mb_k\|_{\mathcal{B}_{\vec{p}\,'}^{p'_0}(\mathbb{R}^n)}
\le C_{n,p_0,\vec{p}}\sum_{k=1}^{\infty}|\lambda_k|
\le 2C_{n,p_0,\vec{p}}\|f\|_{\mathcal{B}_{\vec{p}\,'}^{p'_0}(\mathbb{R}^n)}.
\end{align*}
This completes the proof of Theorem \ref{th5.1}.
\end{proof}

\section{Applications}\label{sec6}
We will prove the boundedness of $[b,I_{\alpha}]$ on mixed Morrey spaces and block spaces in this section.
\begin{theorem}\label{th6.3}
Let $0<\alpha<n,~1<p_0,q_0,\vec{p},\vec{q}<\infty$,
$$1<\vec{p}\le\vec{q}<\infty~ \text{and} ~\alpha=\frac{n}{p_0}-\frac{n}{q_0}=\sum_{i=1}^n\frac{1}{p_i}-\sum_{i=1}^n\frac{1}{q_i}.$$
Then, the following conditions are equivalent:
\begin{itemize}
\item[(\romannumeral1)] $b\in BMO(\mathbb{R}^n)$.
\item[(\romannumeral2)] $[b,I_\alpha]$ is bounded from $\mathcal{M}_{\vec{p}}^{p_0}(\mathbb{R}^n)$ to $\mathcal{M}_{\vec{q}}^{q_0}(\mathbb{R}^n)$.
\end{itemize}
\end{theorem}
\begin{remark}\label{re6.4}
Taking $\frac{n}{p_0}=\sum_{i=1}^n\frac{1}{p_i}$ and $\frac{n}{q_0}=\sum_{i=1}^n\frac{1}{q_i}$, then the result of \cite{N2019(2)} need the condition
$$p_j\sum_{i=1}^n\frac{1}{p_i}=q_j\sum_{i=1}^n\frac{1}{q_i}~~(j=1,\cdots,n).$$
But Theorem \ref{th6.3} does not need this condition.
\end{remark}
\begin{theorem}\label{th6.1}
\begin{itemize}
\item[(\romannumeral1)]Let $0<\alpha<n,~1<p_0,q_0,\vec{p},\vec{q}<\infty$, $\frac{n}{p_0}\le\sum_{i=1}^n\frac{1}{p_i}$ and $\frac{n}{q_0}\le\sum_{i=1}^n\frac{1}{q_i}$. If $I_{\alpha}$ is bounded from $\mathcal{M}_{\vec{p}}^{p_0}(\mathbb{R}^n)$ to $\mathcal{M}_{\vec{q}}^{q_0}(\mathbb{R}^n)$, then for any measurable function $f$
    $$\|I_{\alpha}f\|_{\mathcal{B}_{\vec{p}\,'}^{p'_0}(\mathbb{R}^n)}\lesssim \|f\|_{\mathcal{B}_{\vec{q}\,'}^{q'_0}(\mathbb{R}^n)}.$$
\item[(\romannumeral2)]Let $0<\alpha<n,~1<p_0,q_0,\vec{p},\vec{q}<\infty$, $\frac{n}{p_0}\le\sum_{i=1}^n\frac{1}{p_i}$ and $\frac{n}{q_0}\le\sum_{i=1}^n\frac{1}{q_i}$. If $[b,I_{\alpha}]$ is bounded from $\mathcal{M}_{\vec{p}}^{p_0}(\mathbb{R}^n)$ to $\mathcal{M}_{\vec{q}}^{q_0}(\mathbb{R}^n)$, then for any measurable function $f$
    $$\left\|[b,I_{\alpha}]f\right\|_{\mathcal{B}_{\vec{p}\,'}^{p'_0}(\mathbb{R}^n)}\lesssim \|f\|_{\mathcal{B}_{\vec{q}\,'}^{q'_0}(\mathbb{R}^n)}.$$
\end{itemize}
\end{theorem}

We show some lemmas before our proof. \cite[Lemma 4.7]{CWYZ2020} implies Lemma \ref{ELth}. The \cite[Lemma 2.13]{TYYZ2021} shows Lemma \ref{EIth}. Lemma \ref{SMIth1} shows the sharp maximal theorem on weighted Lebesgue spaces \cite[Theorem 3.4.5]{Gr2009}. Lemma \ref{le6.2} can be refered to \cite[Theorem 1.3]{31} and \cite[Lemma 4.2]{21}. Note that $f$ is locally integrable function in Lemma \ref{le6.2}. Lemma \ref{le6.1} can be found in \cite[Corollary 5.3]{30}.
\begin{lemma}\label{ELth}
Let $X$ be a ball quasi-Banach function space. If there exists $s\in(1,\infty)$ such that $M$ is bounded on $(X^{\frac{1}{s}})'$, then there exists $\varepsilon\in(0,1)$ such that $X$ is continuously embedded into $L^{s}_{\omega}(\mathbb{R}^n)$ with $\omega:=[M(\chi_{Q(0,1)})]^{\varepsilon}\in A_1(\mathbb{R}^n)$, namely, there exists a positive constant $C$ such that, for any $f\in X$,
$$\|f\|_{L_{\omega}^s(\mathbb{R}^n)}\le C\|f\|_{X}.$$
\end{lemma}
\begin{lemma}\label{EIth}
Let $X$ be a ball quasi-Banach function spaces and $s\in(0,\infty)$. Let $\mathcal{F}$ be the set of all pairs of nonnegative measurable functions $(F,G)$ such that, for any given $\omega\in A_1(\mathbb{R}^n)$,
$$\int_{\mathbb{R}^{n_1}\times\mathbb{R}^{n_2}}\left(F(x,y)\right)^{s}\omega(x,y)dxdy\le C_{(s,[\omega]_{A_1(\mathbb{R}^n)})}\int_{\mathbb{R}^{n_1}\times\mathbb{R}^{n_2}}\left(G(x,y)\right)^{s}\omega(x,y)dxdy.$$
where $C_{(s,[\omega]_{A_1(\mathbb{R}^n)})}$ is a positive constant independent of $(F,G)$, but dependents on $s$ and $A_1(\mathbb{R}^n)$. Assume that there exists a $s_0\in[s,\infty)$ such that $X^{\frac{1}{s_0}}$ is ball Banach function space and $Mf$ is bounded on $(X^{\frac{1}{s_0}})'$. Then there exists a positive constant $C_0$ such that, for any $(F,G)\in\mathcal{F}$,
$$\|F\|_{X}\le C_0\|G\|_{X}.$$
\end{lemma}
\begin{lemma}\label{SMIth1}
Let $1<s<\infty$, $\omega\in A_s$. Then for any $f\in L_\omega^s(\R^n)$,
$$\int_{\mathbb{R}^n}\left(Mf(x)\right)^s\omega(x)dx\le C\int_{\mathbb{R}^n}\left(M^{\sharp}f(x)\right)^s\omega(x)dx$$
holds.
\end{lemma}
\begin{remark}
A similar result in \cite[Corollary 3.10]{N2019(2)} requires
\begin{align}\label{Ceq}
Mf\in\mathcal{M}_{\vec{s}}^{s_0}(\R^n)
\end{align}
for some $0<s_0<\infty$ and $\vec{s}=(s_1,s_2,\cdots,s_n)$ with $\frac{n}{s_0}<\sum_{i=1}^{n}\frac{1}{s_i}.$
From Lemma \ref{ELth}, (\ref{Ceq}) means
$$f\in L_{\omega}^s(\mathbb{R}^n)$$
for $1<s<\infty$ and $\omega\in A_1(\mathbb{R}^n)$.
\end{remark}
\begin{lemma}\label{le6.2}
Let $0<\alpha<n$, $1<r<\infty$ and $b\in BMO(\mathbb{R}^n)$. Then for any locally integrable function $f$, there exists a constant $C>0$ independent of $b$ and $f$ such that
$$M^{\sharp}([b,I_{\alpha}](f))(x)\le C\|b\|_{BMO(\mathbb{R}^n)}\left(I_{\alpha}(|f|)(x)+I_{r\alpha}(|f|^r)(x)^{\frac{1}{r}}\right).$$
\end{lemma}
\begin{lemma}\label{le6.1}
Let $0<\alpha<n,~1<p_0,q_0,\vec{p},\vec{q}<\infty$, $\frac{n}{p_0}\le\sum_{i=1}^n\frac{1}{p_i}$ and $\frac{n}{q_0}\le\sum_{i=1}^n\frac{1}{q_i}$. If
$$1<\vec{p}\le\vec{q}<\infty~ \text{and}~\alpha=\frac{n}{p_0}-\frac{n}{q_0}=\sum_{i=1}^n\frac{1}{p_i}-\sum_{i=1}^n\frac{1}{q_i},$$
then there exists a positive $C$ such that
$$\|I_{\alpha}f\|_{\mathcal{M}_{\vec{q}}^{q_0}(\mathbb{R}^n)}\le C\|f\|_{\mathcal{M}_{\vec{p}}^{p_0}(\mathbb{R}^n)}.$$
\end{lemma}
\begin{lemma}\label{SMIth}
Let $1<q_0,\vec{q}<\infty$ and $\frac{n}{q_0}<\sum_{i=1}^{n}\frac{1}{q_i}$. If $f\in L_{\omega}^s(\mathbb{R}^n)$ with $1<s<\infty$ and $\omega\in A_1(\mathbb{R}^n)$, then there exists a positive constant $C$ such that
\begin{equation}\label{eq3.7}
\|Mf\|_{\mathcal{M}_{\vec{q}}^{q_0}(\mathbb{R}^n)}\le C\|M^{\sharp} f\|_{\mathcal{M}_{\vec{q}}^{q_0}(\mathbb{R}^n)}.
\end{equation}
\end{lemma}
\begin{proof}
For any $s\in(0,\infty)$ and $\omega\in A_1(\R^n)$, if $f\in L_{\omega}^s(\mathbb{R}^n)$, then by Lemma \ref{SMIth1},
$$\int_{\mathbb{R}^n}\left(Mf(x)\right)^s\omega(x)dx\le C_{(s,[\omega]_{A_1(\mathbb{R}^n})}\int_{\mathbb{R}^n}\left(M^{\sharp} f(x)\right)^s\omega(x)dx.$$
Taking $s_0\in(1,\min\{s,q_0,q_1,q_2,\cdots,q_n\})$, we have
$$\left(\mathcal{M}_{\vec{q}}^{q_0}(\R^n)\right)^{\frac{1}{s_0}}=\mathcal{M}_{\vec{q}/s_0}^{q_0/s_0}(\R^n)$$
with $M$ is bounded on $\left(\mathcal{M}_{\vec{q}/s_0}^{q_0/s_0}(\R^n)\right)'=\mathcal{B}_{(\vec{q}/s)\,'}^{(q_0/s_0)'}(\mathbb{R}^n)$.
Thus, by Lemma \ref{EIth},
$$\|Mf\|_{\mathcal{M}_{\vec{q}}^{q_0}(\mathbb{R}^n)}\le C\|M^{\sharp} f\|_{\mathcal{M}_{\vec{q}}^{q_0}(\mathbb{R}^n)},$$
for any $f\in L_{\omega}^s(\mathbb{R}^n)$ with $\omega\in A_1(\R^n)$ and $0<s<\infty$.

This completes the proof of Lemma \ref{SMIth}.
\end{proof}
\begin{proof}[Proof of Theorem \ref{th6.3}]
$(\romannumeral1)\Rightarrow(\romannumeral2)$. Let $b\in BMO(\mathbb{R}^n)$. Let $f\in\M_{\vec{p}}^{p_0}(\R^n)$ and $p=\min\{p_1,p_2,\cdots,p_n\}$. From H\"older's inequality, we have
$$f\in\M_{p}^{p_0}(\R^n).$$
Thus, $[b,I_{\alpha}]f\in\M_{q}^{q_0}(\R^n)$ with $\frac{1}{q}=\frac{1}{p}-\frac{\alpha}{n}$.
According to Lemma \ref{ELth}, there exist $1<s<\infty$ such that
$$[b,I_{\alpha}]f\in L^{s}_{\omega}(\R^n),$$
where $\omega=[M(\chi_{Q(0,1)})]^{\varepsilon}$ with $0<\epsilon<1$. Thus, the assumption of Lemma \ref{SMIth} is satisfied.
We use Lemma \ref{SMIth}, Lemma \ref{le6.2} and Lemma \ref{le6.1} to write
\begin{align*}
\left\|[b,I_{\alpha}]f\right\|_{\mathcal{M}_{\vec{q}}^{q_0}(\mathbb{R}^n)}
&\le C\lf\|M([b,I_{\alpha}](f))\r\|_{\mathcal{M}_{\vec{q}}^{q_0}(\mathbb{R}^n)}\\
&\le C\lf\|M^{\sharp}\([b,I_{\alpha}](f)\)\r\|_{\mathcal{M}_{\vec{q}}^{q_0}(\mathbb{R}^n)}\\
&\le C\|b\|_{BMO(\mathbb{R}^n)}\left(\lf\|I_{\alpha}(|f|)\r\|_{\mathcal{M}_{\vec{q}}^{q_0}(\mathbb{R}^n)} +\|I_{r\alpha}(|f|^r)^{\frac{1}{r}}\|_{\mathcal{M}_{\vec{q}}^{q_0}(\mathbb{R}^n)}\right)\\
&=   C\|b\|_{BMO(\mathbb{R}^n)}\left(\lf\|I_{\alpha}(|f|)\r\|_{\mathcal{M}_{\vec{q}}^{q_0}(\mathbb{R}^n)} +\lf\|I_{r\alpha}(|f|^r)\r\|^{\frac{1}{r}}_{\mathcal{M}_{\vec{q}/r}^{q_0/r}(\mathbb{R}^n)}\right)\\
&\le C\|b\|_{BMO(\mathbb{R}^n)}\|f\|_{\mathcal{M}_{\vec{p}}^{p_0}(\mathbb{R}^n)}.
\end{align*}

$(\romannumeral2)\Rightarrow(\romannumeral1)$. Assume that $[b,I_\alpha]$ is bounded from $\mathcal{M}_{\vec{p}}^{p_0}(\mathbb{R}^n)$ to $\mathcal{M}_{\vec{q}}^{q_0}(\mathbb{R}^n)$. We use the same method as Janson \cite{32}. Choose $0\neq z_0\in\mathbb{R}^n$ such that $0\notin Q(z_0,2\sqrt{n})$. Then for $x\in Q(z_0,2\sqrt{n})$, $|x|^{n-\alpha}\in C^{\infty}(Q(z_0,2\sqrt{n}))$. Hence, $|x|^{n-\alpha}$ can be written as the absolutely convergent Fourier series:
\begin{equation}\label{FE}
|x|^{n-\alpha}\chi_{Q(z_0,2\sqrt{n})}(x)=\sum_{m\in \mathbb{Z}^n}a_me^{2im\cdot x}\chi_{Q(z_0,2\sqrt{n})}(x)
\end{equation}
with $\sum_{m\in \mathbb{Z}^n}|a_m|<\infty$.

For any $x_0\in\mathbb{R}^n$ and $t>0$, let $Q=Q(x_0,t)$ and $Q_{z_0}=Q(x_0+z_0t,t)$. Let $s(x)=\overline{sgn(\int_{Q'}(b(x)-b(y))dy)}$, then
\begin{align*}
\frac{1}{|Q|}\int_Q|b(x)-b_{Q_{z_0}}|
&=\frac{1}{|Q|}\frac{1}{|Q_{z_0}|}\int_Q\left|\int_{Q_{z_0}}(b(x)-b(y))dy\right|dx\\
&=\frac{1}{|Q|}\frac{1}{|Q_{z_0}|}\int_Q\int_{Q_{z_0}}s(x)(b(x)-b(y))dydx\\
&=t^{-2n}\int_Q\int_{Q_{z_0}}s(x)(b(x)-b(y))dydx.
\end{align*}
If $x\in Q$ and $y\in Q_{z_0}$, then $\frac{y-x}{t}\in Q(z_0,2\sqrt{n})$. Hence, (\ref{FE}) shows that
\begin{align*}
\frac{1}{|Q|}\int_Q|b(x)-b_{Q_{z_0}}|
&=t^{-n-\alpha}\int_Q\int_{Q_{z_0}}s(x)(b(x)-b(y))|x-y|^{\alpha-n}\left(\frac{|x-y|}{t}\right)^{n-\alpha}dydx\\
&=t^{-n-\alpha}\sum_{m\in\mathbb{Z}^n}a_m\int_Q\int_{Q_{z_0}}s(x)(b(x)-b(y))|x-y|^{\alpha-n}e^{-2im\cdot \frac{y}{t}}dy\times e^{2im\cdot \frac{x}{t}}dx\\
&=t^{-n-\alpha}\sum_{m\in\mathbb{Z}^n}a_m\int_Q[b,I_\alpha](e^{-2im\cdot \frac{\cdot}{t}}\chi_{Q_{z_0}})(x)\times s(x)e^{2im\cdot \frac{x}{t}}dx.
\end{align*}

By the definition of K\"othe dual spaces,
$$\frac{1}{|Q|}\int_Q|b(x)-b_{Q_{z_0}}|
\le t^{-n-\alpha}\sum_{m\in\mathbb{Z}^n}a_m\lf\|[b,I_\alpha](e^{-2im\cdot \frac{\cdot}{t}}\chi_{Q_{z_0}})\r\|_{\mathcal{M}_{\vec{q}}^{q_o}(\mathbb{R}^n)} \lf\|s\cdot e^{-2im\cdot \frac{\cdot}{t}}\chi_Q\r\|_{\mathcal{B}_{\vec{q}\,'}^{q'_0}(\mathbb{R}^n)}.$$
It is easy to calculate
$$\lf\|s\cdot e^{-2im\cdot \frac{\cdot}{t}}\chi_Q\r\|_{\mathcal{B}_{\vec{q}\,'}^{q'_0}(\mathbb{R}^n)} =\lf\|\chi_Q\r\|_{\mathcal{B}_{\vec{q}\,'}^{q'_0}(\mathbb{R}^n)}\lesssim t^{n-\frac{n}{{q_0}}}.$$
Hence,
$$\frac{1}{|Q|}\int_Q|b(x)-b_{Q_{z_0}}| =t^{-\frac{1}{q_0}-\alpha}\sum_{m\in\mathbb{Z}^n}a_m\left\|[b,I_\alpha](e^{-2im\cdot \frac{\cdot}{t}}\chi_{Q_{z_0}})\right\|_{\mathcal{M}_{\vec{q}}^{q_0}(\mathbb{R}^n)}.$$
According to the hypothesis
\begin{align*}
&\quad\frac{1}{|Q|}\int_Q|b(x)-b_{Q_{z_0}}|\\
&\le t^{-\frac{1}{q_0}-\alpha}\sum_{m\in\mathbb{Z}^n}a_m\lf\|e^{-2im\cdot \frac{\cdot}{t}}\chi_{Q_{z_0}}\r\|_{\mathcal{M}_{\vec{p}}^{p_0}(\mathbb{R}^n)} \lf\|[b,I_\alpha]\r\|_{\mathcal{M}_{\vec{p}}^{p_0}(\mathbb{R}^n)\rightarrow \mathcal{M}_{\vec{q}}^{q_0}(\mathbb{R}^n)}\\
&=t^{-\frac{1}{q_0}-\alpha+\frac{1}{p_0}} \sum_{m\in\mathbb{Z}^n}a_m\lf\|[b,I_\alpha]\r\|_{\mathcal{M}_{\vec{p}}^{p_0}(\mathbb{R}^n)\rightarrow \mathcal{M}_{\vec{q}}^{q_0}(\mathbb{R}^n)}\\
&\le \sum_{m\in\mathbb{Z}^n}|a_m|\|[b,I_\alpha]\|_{\mathcal{M}_{\vec{p}}^{p_0}(\mathbb{R}^n)\rightarrow \mathcal{M}_{\vec{q}}^{q_0}(\mathbb{R}^n)}
\le C\|[b,I_\alpha]\|_{\mathcal{M}_{\vec{p}}^{p_0}(\mathbb{R}^n)\rightarrow \mathcal{M}_{\vec{q}}^{q_0}(\mathbb{R}^n)}.
\end{align*}
Thus, we have
$$\frac{1}{|Q|}\int_Q|b(x)-b(y)|dx\le\frac{2}{|Q|}\int_Q|b(x)-b_{Q_{z_0}}|dx\le C\|[b,I_\alpha]\|_{L^{\vec{p}}(\mathbb{R}^n)\rightarrow L^{\vec{q}}(\mathbb{R}^n)}$$
This prove $b\in BMO(\mathbb{R}^n)$.

The proof of Theorem \ref{th6.3} is complete.
\end{proof}

\begin{proof}[Proof of Theorem \ref{th6.1}]
(\romannumeral1) Let $U_{\vec{p}}^{p_0}$ denote the unite ball on $\mathcal{M}_{\vec{p}}^{p_0}(\mathbb{R}^n)$.
Suppose that $I_{\alpha}$ is bounded from $\mathcal{M}_{\vec{p}}^{p_0}(\mathbb{R}^n)$ to $\mathcal{M}_{\vec{q}}^{q_0}(\mathbb{R}^n)$. By Definition \ref{def2.3}
$$\lf\|I_{\alpha}f\r\|_{\mathcal{B}_{\vec{p}\,'}^{p'_0}(\mathbb{R}^n)}
=\sup_{g\in U_{\vec{p}}^{p_0}} \int_{\mathbb{R}^n}|g(x)I_{\alpha}f(x)|dx.$$
In view of (\ref{eq1.1}), this implies
$$\|I_{\alpha}f\|_{\mathcal{B}_{\vec{p}\,'}^{p'_0}(\mathbb{R}^n)}
\le\sup_{g\in U_{\vec{p}}^{p_0}} \int_{\mathbb{R}^n}|f(x)|I_{\alpha}(|g|)(x)dx.$$
Therefore, by the assumption and (\romannumeral2) of Remark \ref{re2.4},
\begin{align*}
\|I_{\alpha}f\|_{\mathcal{B}_{\vec{p}\,'}^{p'_0}(\mathbb{R}^n)}
&\le\|f\|_{\mathcal{B}_{\vec{q}\,'}^{q'_0}(\mathbb{R}^n)}
\sup_{g\in U_{\vec{p}}^{p_0}}
\lf\|I_{\alpha}(|g|)\r\|_{\mathcal{M}_{\vec{q}}^{q_0}(\mathbb{R}^n)}
\lesssim \|f\|_{\mathcal{B}_{\vec{q}\,'}^{q'_0}(\mathbb{R}^n)}.
\end{align*}
The inverse is similar to the above, so the details are omitted.

(\romannumeral2) By the same way as (\romannumeral1) and (\romannumeral4) of Remark \ref{re2.4},
\begin{align*}
\|[b,I_\alpha]f\|_{\mathcal{B}_{\vec{p}\,'}^{p'_0}(\mathbb{R}^n)}
&=\sup_{g\in U_{\vec{p}}^{p_0}} \left|\int_{\mathbb{R}^n}g(x)[b,I_\alpha]f(x)dx\right|\\
&=\sup_{g\in U_{\vec{p}}^{p_0}} \left|\int_{\mathbb{R}^n}f(x)[b,I_\alpha](g)(x)dx\right|\\
&\le\|f\|_{\mathcal{B}_{\vec{q}\,'}^{q'_0}(\mathbb{R}^n)}
\sup_{g\in U_{\vec{p}}^{p_0}}
\|[b,I_\alpha](g)\|_{\mathcal{M}_{\vec{q}}^{q_0}(\mathbb{R}^n)}\\
&\lesssim \|f\|_{\mathcal{B}_{\vec{q}\,'}^{q'_0}(\mathbb{R}^n)}.
\end{align*}
The inverse is similar to the preceeding argument, so we omit it. The proof of Theorem \ref{th6.1} is complete.
\end{proof}

\bigskip
\noindent Houkun Zhang\\
College of Mathematics and System Sciences\\
Xinjiang University\\
Urumqi 830046, China\\
\smallskip
\noindent\emph{Email address}: \texttt{zhanghkmath@163.com}\\

\noindent Jiang Zhou\\
College of Mathematics and System Sciences\\
Xinjiang University\\
Urumqi 830046, China\\
\smallskip
\noindent\emph{Email address}: \texttt{zhoujiang@xju.edu.cn}\\
\end{document}